\documentclass[letterpaper, 10 pt, conference]{ieeeconf}  % Comment this line out if you need a4paper

\IEEEoverridecommandlockouts                              % This command is only needed if 
\usepackage[T1]{fontenc}    % use 8-bit T1 fonts
\usepackage{hyperref}       % hyperlinks
\usepackage{url}            % simple URL typesetting
\usepackage{booktabs}       % professional-quality tables
\usepackage{amsfonts}       % blackboard math symbols
\usepackage{nicefrac}       % compact symbols for 1/2, etc.
\usepackage{microtype}      % microtypography
\usepackage{amsmath}
\usepackage{amssymb}
\usepackage{comment}
\usepackage{pdfpages}
\usepackage{subcaption}
\usepackage{lipsum} % For generating placeholder text
\usepackage{booktabs} % For better table formatting
\usepackage{adjustbox} % To fit the table on a page
\def\argmin{\mathop{\rm argmin}}
\usepackage{afterpage}
\usepackage[mathscr]{euscript}
\usepackage{algorithm}
\usepackage{algorithmic}
\usepackage{stfloats}
\usepackage{multirow}
\usepackage{adjustbox}
\usepackage{booktabs}
\usepackage{multicol}
\newtheorem{theorem}{Theorem}%[section]

\newtheorem{lemma}{Lemma}%[section]
\newtheorem{definition}{Definition}%[section]
\newtheorem{remark}{Remark}%[section]
\newtheorem{assumption}{Assumption}%[section]
\newtheorem{proposition}{Proposition}%[section]
\usepackage{array,multirow,graphicx}
\usepackage{float}
\usepackage{caption}
\title{\LARGE \bf
Achieving optimal complexity guarantees for a class of bilevel convex optimization problems 
}

\author{Sepideh Samadi$^{1}$,  Daniel Burbano$^{2}$, and Farzad Yousefian$^{1}$% <-this % stops a space
\thanks{*This work was funded in part by the NSF under CAREER grant ECCS-$1944500$, in part by the ONR under grant N$00014$-$22$-$1$-$2757$, and in part by the DOE under grant DE-SC$0023303$.}% <-this % stops a space
\thanks{$^{1}$Samadi and Yousefian are with the Department of Industrial and Systems Engineering, Rutgers University, Piscataway 08854, NJ, USA.
        {\tt\small sepideh.samadi@rutgers.edu, farzad.yousefian@rutgers.edu}. $^{2}$Burbano is with the Department of Electrical and Computer Engineering, Rutgers University, Piscataway 08854, NJ, USA.
        {\tt\small daniel.burbano@rutgers.edu}.}%
}

\begin{document}

\maketitle
\thispagestyle{empty}
\pagestyle{empty}

%%%%%%%%%%%%%%%%%%%%%%%%%%%%%%%%%%%%%%%%%%%%%%%%%%%%%%%%%%%%%%%%%%%%%%%%%%%%%%%%
\begin{abstract}
We design and analyze a novel accelerated gradient-based algorithm for a class of bilevel optimization problems. These problems have various applications arising from machine learning and image processing, where optimal solutions of the two levels are interdependent. That is, achieving the optimal solution of an upper-level problem depends on the solution set of a lower-level optimization problem.
 We significantly improve existing iteration complexity to $\mathcal{O}(\epsilon^{-0.5})$ for both suboptimality and infeasibility error metrics, where $\epsilon>0$ denotes an arbitrary scalar. In addition, contrary to existing methods that require solving the optimization problem sequentially (initially solving an optimization problem to approximate the solution of the lower-level problem followed by a second algorithm), our algorithm concurrently solves the optimization problem.  To the best of our knowledge, the proposed algorithm has the fastest known iteration complexity, which matches the optimal complexity for single-level optimization. We conduct numerical experiments on sparse linear regression problems to demonstrate the efficacy of our approach. 
\end{abstract}

%%%%%%%%%%%%%%%%%%%%%%%%%%%%%%%%%%%%%%%%%%%%%%%%%%%%%%%%%%%%%%%%%%%%%%%%%%%%%%%%
\section{Introduction}\label{sec:intro}

In this paper, we consider a class of bilevel convex optimization problems of the form
\begin{equation}
\min \  f(x),\quad \textrm {s.t.} \quad  x \in \argmin_{x \in \mathbb{R}^n} \ \{\bar{h}(x)\triangleq  h(x) + \omega(x)  \},
\label{prob:centr}
\end{equation}
where $f: \mathbb{R}^n \rightarrow  \mathbb{R} $ is a smooth convex function and $\bar{h}(x)\triangleq  h(x) + \omega(x)$  with $h: \mathbb{R}^n \rightarrow  \mathbb{R} $ and $w: \mathbb{R}^n \rightarrow  \mathbb{R} $ being smooth convex and nonsmooth convex functions, respectively.
Note that the optimization problem in equation \eqref{prob:centr} can be written as the minimization of the upper-level function $f(\cdot)$ over the set $X^*$ as 
\begin{equation} \label{prob:main 2} 
\min \  f(x),\quad\textrm {s.t.} \quad  x \in X^*,
\end{equation}
where $X^*$ is the nonempty convex set of all optimal solutions of the composite inner optimization problem given by
\begin{equation} \label{prob:main 3} 
\min \  \bar{h}(x)\triangleq  h(x) + \omega(x),\quad \textrm {s.t.} \quad  x \in \mathbb{R}^n.
\end{equation}
The optimization problem in \eqref{prob:centr} or equivalently in \eqref{prob:main 2}-\eqref{prob:main 3} is known as {\it the optimal solution selection problem} and also, {\it the simple bilevel optimization problem}. 
This problem has gained attention in the last few years due to its emergence in addressing over-parameterized machine learning, ill-posed optimization in image processing, risk-averse learning, and portfolio optimization~\cite{friedlander2008exact, beck2014first,amini2019iterative,merchav2023convex} among others. 
The pressing challenge consists of developing fast iterative algorithms that solve the optimization problem in \eqref{prob:main 2}-\eqref{prob:main 3}.

Within the existing literature on single-level composite optimization problems, it is well-known that the accelerated proximal method (APM) or FISTA (Fast Iterative Shrinkage-Thresholding Algorithm) admits the fastest iteration complexity $\mathcal{O}(\epsilon^{-0.5})$~\cite{beck2009fast}. However, to the best of our knowledge, no existing optimization methods for simple bilevel problems have been reported to reach a comparable convergence rate to the single-level problem. Motivated by this gap in the literature, here, we develop a novel method (R-APM: regularized accelerated proximal method) that solves the bilevel problem with a rate of $\mathcal{O}(\epsilon^{-0.5})$. Specifically, the proposed algorithm leverages the structure of the FISTA method by incorporating a regularized gradient mapping that depends on the weighted information of the gradient of the lower-level (only the smooth component) and upper-level objective functions, respectively.

\begin{table*}[t]
  \centering
  \caption{Comparison between our proposed method and the most relevant existing strategies.}
\label{table:lit}
  \begin{adjustbox}{width=\textwidth}
    \begin{tabular}{|c|c|c|c|c|c|}
      \hline
      \multirow{2}{*}{Reference} & \multirow{2}{*}{Method} & \multicolumn{2}{c|}{Objective Functions} & \multicolumn{2}{c|}{Iteration Complexity} \\
      \cline{3-6}
      & & Upper level & Lower level &  Suboptimality & Infeasibility \\
      \hline
(Solodov) \cite{solodov2007explicit}& Explicit Descent Method&Convex, Smooth &Convex, Smooth  &Asymptotic &Asymptotic  \\
 (Solodov) \cite{solodov2007bundle}& Bilevel Bundle Method&Convex, Nonsmooth & Convex, Nonsmooth &Asymptotic &Asymptotic  \\
(Beck et al.) \cite{beck2014first} &  MNG & Strongly Convex, Differentiable & Convex, Smooth &  Asymptotic & $\mathcal{O}(\epsilon^{-2})$  \\
(Helou et al.) \cite{helou2017}&FIBA &Convex& Convex, Smooth&Asymptotic & Asymptotic \\
(Malitsky et al.) \cite{malitsky2017chambolle} & Tseng’s Method &  Convex& Convex& Asymptotic & $\mathcal{O}(\epsilon^{-1})$  \\
     (Sabach et al.) \cite{sabach2017first} & BiG-SAM & Strongly Convex, Smooth & Convex &  Asymptotic & $\mathcal{O}(\epsilon^{-1})$  \\
(Amini et al.) \cite{amini2019iterative}&IR-IG&Convex, Nonsmooth&Strongly Convex, Nonsmooth&  $\mathcal{O}(\epsilon^{-\frac{1}{0.5-b}}), b \in (0,0.5)$  & Asymptotic \\
     (Kaushik et al.) \cite{kaushik2021method} & a-IRG &  Convex, Nonsmooth &  Convex, Nonmooth & $\mathcal{O}(\epsilon^{-4})$ & $\mathcal{O}(\epsilon^{-4})$ \\
     (Jiang et al.) {\cite{jiang2022conditional}} &{CG-BiO} &{ Convex, Smooth}&{ Convex, Smooth} & $\mathcal{O}(\epsilon^{-1})$ &$\mathcal{O}(\epsilon^{-1})$ \\
(Merchav et al.) \cite{merchav2023convex}& Bi-SG&Convex & Convex &  $\mathcal{O}(\epsilon^{-\frac{1}{1-\alpha}})$  & $\mathcal{O}(\epsilon^{-\frac{1}{\alpha}}), \alpha \in (0.5,1)$    \\
(Shen et al.) \cite{shen2023online}& Structure Exploiting Algorithm&Convex&Convex&$\mathcal{O}(\epsilon^{-2})$&$\mathcal{O}(\epsilon^{-2})$ \\
      \hline
      {\bf This work} & R-APM & Convex, Smooth &  Convex, Nonsmooth  (composite)& $\mathcal{O}(\epsilon^{-0.5})$ & $\mathcal{O}(\epsilon^{-0.5})$  \\
      \hline
    \end{tabular}
  \end{adjustbox}
\end{table*}
While the literature on bilevel optimization methods is not as extensive as that for single-level problems, there have been several notable efforts in this area since the seminal work of Tikhonov~\cite{tikhonov1963solution}. Specifically, in \cite{solodov2007explicit}, the explicit descent method was proposed to solve the optimization problem in \eqref{prob:centr}, assuming that the upper and lower-level functions are both convex. Their theoretical analysis, however, indicated that the algorithm exhibits asymptotic convergence. This result was further extended in~\cite{solodov2007bundle} to the case where the upper and lower-level functions are non-smooth. Similarly, in \cite{friedlander2008exact}, the authors studied asymptotic convergence properties of bilevel optimization algorithms, while in \cite{helou2017}, the authors introduced the $\epsilon$-subgradient method to solve simple bilevel problems and studied its asymptotic convergence. Specifically, the authors considered the upper-level objective function to be convex. They introduced two different algorithms, namely, the Fast Iterative Bilevel Algorithm (FIBA) and Incremental Iterative Bilevel Algorithm (IIBA), that consider smooth and non-smooth lower-level objective functions, respectively.
Alternatively, the Minimal Norm Gradient (MNG) was introduced in \cite{beck2014first}, assuming the upper-level objective function is smooth and strongly convex. The authors found that the lower-level objective function reaches the iteration complexity $\mathcal{O}(\epsilon^{-2})$. This rate was later improved to $\mathcal{O}(\epsilon^{-1})$ as reported in~\cite{sabach2017first}. 
In \cite{malitsky2017chambolle}, the authors considered convex functions for both levels and obtained an iteration complexity of $\mathcal{O}(\epsilon^{-1})$ for infeasibility, while in \cite{amini2019iterative} the Iterative Regularized Incremental projected (sub)Gradient (IR-IG) algorithm was introduced considering the upper-level is strongly convex.    
More recent progress in addressing this class of bilevel problems include the averaging Iteratively Regularized Gradient (a-IRG) introduced in \cite{kaushik2021method}. In this study, the authors provided a complexity of $\mathcal{O}(\epsilon^{-4})$ for both suboptimality and infeasibility metrics, in the case where the lower-level problem is a variational inequality problem (capturing optimization as an instance). In distributed regimes, regulation-based gradient tracking method was proposed in \cite{yousefian2021bilevel} to solve a distributed convex simple bilevel optimization problem over networks.
In \cite{gong2021bi}, the Dynamic Barrier Gradient Descent (DBGD) methods was introduced for both constrained and lexicographic optimization and provided convergence with general nonconvex functions. 
\cite{shen2023online} provided convergence rates for the convex simple bilevel optimization problem by leveraging the duality theory.

Bi-Sub-Gradient (Bi-SG) method was proposed in \cite{merchav2023convex} for convex simple bilevel optimization problem with nonsmooth upper-level objective function and iteration complexity of $\mathcal{O}(\epsilon^{-\frac{1}{1-\alpha}})$ and $\mathcal{O}(\epsilon^{-\frac{1}{\alpha}}), \alpha \in (0.5,1)$ were derived for suboptimality and infeasibility, respectively. In addition, \cite{jiang2022conditional} obtained convergence rates for suboptimality and infeasibility of both convex and nonconvex upper-level functions with convex lower-level settings by introducing the CG-Based Bilevel Optimization (CG-BiO). When both functions are convex, CG-BiO achieves an iteration complexity of $\mathcal{O}(\epsilon^{-1})$. Also, when the upper-level function is nonconvex, $\mathcal{O}({\epsilon_{f}}^{-2})$ and $\mathcal{O}(\frac{1}{\epsilon_{f}\epsilon_h})$ were obtained for the suboptimality and infeasibility metrics, respectively. 
 
When compared to existing works, our proposed method achieves the best-known convergence rate as illustrated in Table \ref{table:lit}.
Specifically, the contributions are twofold.
\begin{itemize}
\item[(i)] We propose a novel algorithm termed R-APM to solve simple bilevel problems. We conduct a thorough theoretical analysis and provide convergence rates in terms of infeasibility and suboptimality metrics. Importantly, the complexity $\mathcal{O}(\epsilon^{-0.5})$ matches that of single-level optimization which is known to exhibit the fastest speed.

\item[(ii)] The proposed method works in a concurrent manner. That is, unlike sequential algorithms that require initially solving an optimization problem to approximate the solution of the lower-level problem followed by a second algorithm~\cite{jiang2022conditional}, our algorithm does not require such an approximation. The concurrent nature of our algorithm renders it computationally efficient and makes it suitable for real-time applications where the objective function might slowly vary over time.
%and provided convergence rates in terms of infeasibility and suboptimality metrics.
%
\end{itemize}
The remainder of this paper is organized as follows. In Section~\ref{Sec: 2}, we discuss the outline of  presented algorithm and its main assumptions and required preliminaries for addressing problem~\eqref{prob:centr}. The convergence analysis regarding infeasibility and suboptimality are provided in Section \ref{section 3}. 
Section \ref{sec: 4} includes numerical experiments and the comparisons with some other existing methods. At the end, concluding remarks appear in Section \ref{section 5}.

%As it is shown in Problem \eqref{prob:centr}, we have a composite optimization problem in the lower level which has two smooth and %nonsmooth functions whose some examples are presented in the following. The first one is convex-constrained smooth minimization, which is 
%\begin{equation} \label{prob:main 4} 
%\begin{aligned}
%& \min \  h(x) \\ & \textrm{s.t.} \quad  x \in X.
%\end{aligned}
%\end{equation}

%Now, if we let  $w= \delta_X$ as an indicator function over nonempty, closed, convex set $X$, then we can reformulate \eqref{prob:main 4} to the compact optimization problem \eqref{prob:main 3}. 
%The second example is $\ell_1$-regularized minimization that is \eqref{prob:main 3} while function $w = \lambda \|x\|_1$ for some $\lambda >0$. This example is very popular in machine learning problems to meet sparsity in the solution.

%\noindent (ii) Distributed methods.
%==================================================
%==========       Section II ======================

\section{Mathematical preliminaries}\label{Sec: 2}
\subsection{Notation}
For $x, y \in \mathbb{R}^n$, we use $\langle x, y\rangle$ for their inner product. For any vector $x \in \mathbb{R}^n$, let $\|x\|$ denote the $\ell_2$-norm. For any $x\in \mathbb {R}$, we let $ |x |$ denote absolute value of $x$. The proximal map of function  $g: \mathbb{R}^n \rightarrow (-\infty, \infty]$ for any $x \in \mathbb{R}^n$ is denoted by $\text{prox}_{g}[x]$ and is formally defined in Definition~\ref{def1}.
We let $\nabla f(x) \in \mathbb{R}^n$ be the gradient of the function  $f: \mathbb{R}^n \rightarrow  \mathbb{R} $, with the domain $\mbox{dom}(f)$ at point $x \in \mbox{dom}(f)$. For a convex function $f: \mathbb{R}^n \rightarrow  \mathbb{R} $,  a vector $\widetilde{\nabla}f(x) \in \mathbb{R}^n$ is called a subgradient of $f$ at $x$ if $f(x) + \langle\widetilde{\nabla}f(x),y-x\rangle \leq f(y) $ for all $y \in \mbox{dom}(f)$. Also, we let $\partial f(x) $ denote the subdifferential set of this function $x \in \mbox{dom}(f)$. We let $\Pi_{X}(x)$ denote Euclidean projection operator of vector $x$ onto the set $X$. Also, $\mbox{dist}(x, X) = \|x- \Pi_{X}(x) \|$ represents the distance of vector $x$ from the nonempty closed convex set $X$.

\begin{definition}\label{def1}\em
Given a function $g: \mathbb{R}^n \rightarrow (-\infty, \infty]$, the proximal map of $g$ is defined for all $x\in \mathbb{R}^n$ as
	\begin{align*}
	\text{prox}_{g}[x] \triangleq \underset{u \in \mathbb{R}^n}{\text{argmin}} \ \{  g(u) + \tfrac{1}{2}\|u-x\|^2  \}.
	\end{align*}
\end{definition}

\begin{lemma}[{\cite[Theorem 6.39]{Beck2017}}]\label{lem:RAPM1} 
Let $g: \mathbb{R}^n \rightarrow (-\infty, \infty]$ be proper, closed, and convex. Then for any $u\in \mathbb{R}^n$ and $\gamma > 0$, $(u-z) \in \gamma \partial g(z)$ if and only if $z\triangleq \text{prox}_{\gamma g}[u].$
\end{lemma}
In some of the rate results, we utilize the weak sharpness property, defined as follows.
\begin{definition}[Weak sharp minima{~\cite[Definition 1.1]{studniarski1999weak}}]\label{def:weaksharp}
Consider the inner optimization problem in \eqref{prob:centr}. We say $X^*_{\bar h} \triangleq \hbox{arg}\min_{x\in \mathbb{R}^n}\bar{h}(x)$ is $\alpha$-weakly sharp if there exists $\alpha>0$ such that for all $x \in \mathbb{R}^n$, we have 
$$\bar{h}(x) - \inf_{x\in \mathbb{R}^n} \bar{h}(x) \geq \alpha\  \mbox{dist}(x, X^*_{\bar h}).$$
\end{definition}

\subsection{Regularized optimization problem}
Throughout, we consider the bilevel problem in \eqref{prob:centr} under the following main assumptions. 
\begin{assumption}\label{assump:accel} Consider problem \eqref{prob:centr}.

\noindent (i) Functions $ f, h,$ and $\omega: \mathbb{R}^n \rightarrow (-\infty, \infty]$ are proper, closed, and convex functions.

 \noindent (ii) $f$ and $h$ are $L_f$-smooth and $L_h$-smooth, respectively. 

 \noindent (iii) The set $\argmin_{x \in \mathbb{R}^n} \{ h(x) + \omega(x)  \}$ is nonempty.
\end{assumption}
We define the following functions for all $x\in\mathbb{R}^n$ as 
\begin{align}
     &f_{\eta}(x) \triangleq h(x) + \eta f(x),  \\
    &F_{\eta}(x) \triangleq  f_{\eta}(x)+ \omega(x) , \\
    &q(x) \triangleq  \text{prox}_{\gamma {\omega}}\left[x - \gamma \left(\nabla h (x) + \eta \nabla f(x)\right)\right],
\end{align} 
where $\eta>0$ and $\gamma>0$ are two positive scalars.
Then, the regularized optimization problem \eqref{prob:centr} is given by
\begin{equation} \label{prob:reg} 
%\tag{$1$}  
\begin{split}
\min \  f_{\eta}(x)+ \omega(x),\quad \textrm {s.t.} \quad  x \in \mathbb{R}^n.
\end{split}
\end{equation}

Next, we provide an important result for the regularized function $F_\eta$ that will be crucial when analyzing the convergence of the proposed algorithm.   
\begin{lemma}\label{lem:RAPM2}
	Let Assumption \ref{assump:accel} hold. Let $\gamma,\eta > 0$ be  such that $\gamma \leq {1}/\left({L}_h+\eta{L}_f\right) $. Then, for any $x, y \in \mathbb{R}^n$ 
 \begin{align*}%\label{eqn:lemRAPM1}
F_{\eta}(x) - F_{\eta}(q(y)) \geq \frac{\langle y-x,q(y)-y\rangle }{\gamma} + \frac{\| q(y)-y \|^2}{2\gamma}.
\end{align*} 
\end{lemma}
\begin{proof}
Note that the function $f_\eta$ is  $L_{\eta}$-smooth where $L_{\eta}\triangleq {L}_h+\eta{L}_f$. Thus, for any $y\in\mathbb{R}^n$ we have
\begin{align*}
	f_\eta(q(y)) \leq f_\eta(y) + \langle q(y)-y,\nabla f_\eta(y)	\rangle+ \tfrac{L_\eta}{2}\| q(y)-y\|^2.
\end{align*}
Also, from the convexity of $f_\eta$, for any $x\in\mathbb{R}^n$ we have 
\begin{align*}
	  f_\eta(y) 	 \leq f_\eta(x) -\langle x-y,\nabla f_\eta(y)\rangle.
\end{align*}
Summing the preceding relations, we obtain 
\begin{align}\label{eqn:smooth_f_eta_1}
	\langle x-q(y),\nabla f_\eta(y)	\rangle- \tfrac{L_\eta}{2}\| q(y)-y\|^2 \leq f_\eta(x) -f_\eta(q(y)) .
\end{align}
From Lemma~\ref{lem:RAPM1} and the definition of $q(y)$, we have $y - \gamma \nabla f_\eta(y) -q(y) \in \gamma \partial \omega(q(y)).$ Using this, from the definition of subgradients for $\omega$, we have
\begin{align}\label{eqn:smooth_f_eta_2}
 \langle x-q(y),\tfrac{y -q(y)}{\gamma}-  \nabla f_\eta(y)\rangle \leq \omega(x) - \omega(q(y)).
\end{align}
Summing \eqref{eqn:smooth_f_eta_1} and \eqref{eqn:smooth_f_eta_2} and invoking $\gamma \leq \frac{1}{L_\eta}$, we obtain 
\begin{align*}
	\tfrac{1}{\gamma}\langle x-q(y),y -q(y)	\rangle- \tfrac{1}{2\gamma}\| q(y)-y\|^2 \leq F_\eta(x) -F_\eta(q(y)) .
\end{align*}
This implies the desired inequality. 
\end{proof}

%In this section, we present our main assumptions on Problem \eqref{prob:centr} and the outline of the developed algorithm to address this problem. Also, we provide some preliminary results which are used in the next Section \eqref{section 3}.
\section{The algorithm}
Motivated by the widespread adoption of FISTA~\cite{beck2009fast} as a highly effective approach for solving single-level composite optimization problems, we introduce a novel algorithm termed the Regularized Accelerated Proximal Method (R-APM), as outlined in Algorithm~\ref{alg:RAPM}, to solve the simple bilevel problem in \eqref{prob:centr}.
Our algorithm can be seen as a regularized adaptation of the FISTA method. The primary distinction in our R-APM method lies in the use of regularization in step \#3, where we incorporate the information of both upper and lower-level information as $\nabla h( y_{k}) + \eta {\nabla} f( y_{k})$ leveraging regularization. 
The constant parameter $\eta$ plays a critical role in deriving the convergence rates. An important step in our analysis will be centered on obtaining an explicit value of this parameter to guarantee fast convergence. In fact, as it will be shown in Theorem~\ref{thm:RAPM}, when the lower-level solution set admits a weakly-sharp property, if $\eta$ is chosen below a prescribed threshold, then the optimal complexity $\mathcal{O}(\epsilon^{-0.5})$ is achieved for both of the levels. This is a significant improvement to the complexity of $\mathcal{O}(\epsilon^{-1})$ in the recent work~\cite{jiang2022conditional} that is guaranteed under a similar set of assumptions. 
\begin{algorithm}   
 	\caption{Regularized Accelerated Proximal Method\\(R-APM)}
 	\begin{algorithmic}[1]\label{alg:RAPM}
 		\STATE {{\bf Input:} $y_1=x_0 \in \mathbb{R}^n$, $t_1=1$, scalars $\gamma,\eta > 0$ such that $\gamma  \leq 1/ \left(L_h + \eta L_f\right)$, $K\geq 1$.}
 		\FOR {$k = 1,2, \dots, {K}$}
 		\STATE   {$x_k := \text{prox}_{\gamma {\omega}}\left[ y_{k} -\gamma \left(\nabla h( y_{k}) + \eta {\nabla} f( y_{k})\right)\right]$}	
        \STATE {$t_{k+1} := 0.5+\sqrt{0.25+t_k^2}$}
        \STATE {$y_{k+1} := x_k + (t_k -1)(x_k-x_{k-1})/t_{k+1}$}
 		\ENDFOR
	\STATE {{\bf return:} $x_K$}
 	\end{algorithmic}
  \label{alg:1}
 \end{algorithm}

The next result provides a recursive inequality that will be used to prove the main convergence results of Algorithm \ref{alg:1}. The main steps in the proof are similar to those of Lemma~4.1 in \cite{beck2009fast}.
\begin{lemma}\label{lem:RAPM3}
	Consider Algorithm~\ref{alg:RAPM} and let Assumption~\ref{assump:accel} hold. Then, for any $x \in \mathbb{R}^n$ and $k\geq 1$ we have
 \begin{align*} 
\tfrac{2}{L_{\eta}}\left(t_k^2v_k^2(x)-t_{k+1}^2v_{k+1}^2(x)\right) \geq \|u_{k+1}(x)\|^2-\|u_{k}(x)\|^2,
\end{align*}
where we define for $k\geq 0$, $v_k(x)\triangleq F_\eta(x_k) - F_\eta(x)$ and $u_k(x)\triangleq t_kx_k -  (t_k-1)x_{k-1} -x$.
\end{lemma}
\begin{proof}
Note that $x_{k+1}= q(y_{k+1})$. Thus, invoking Lemma~\ref{lem:RAPM2} for $x:=x_k$ and $ y:=y_{k+1}$, we obtain 
\begin{align*}
    2L_{\eta}^{-1}(v_k(x)-v_{k+1}(x)) &\geq \|x_{k+1}-y_{k+1}\|^2 \\ &+2\langle x_{k+1}-y_{k+1},y_{k+1}-x_k\rangle.
\end{align*} 
Invoking Lemma~\ref{lem:RAPM2} for $x:=x, y:=y_{k+1}$, we obtain 
$$-2L_{\eta}^{-1}v_{k+1}(x)\geq \|x_{k+1}-y_{k+1}\|^2 +2\langle x_{k+1}-y_{k+1},y_{k+1}-x\rangle. $$
Multiplying both sides of the first inequality by $t_{k+1}-1$ and adding it to the second inequality, we obtain 
\begin{align*}
     &2L_{\eta}^{-1}\left((t_{k+1}-1)v_k-t_{k+1}v_{k+1}\right) \geq t_{k+1}\|x_{k+1}-y_{k+1}\|^2 \\
    & +2 \langle x_{k+1}-y_{k+1},t_{k+1}y_{k+1}-(t_{k+1}-1)x_k-x \rangle.
\end{align*}
Multiplying the preceding relation by $t_{k+1}$, we obtain 
\begin{align*}
     &2L_{\eta}^{-1}\left(t_{k}^2v_k-t_{k+1}^2v_{k+1}\right) \geq \|t_{k+1}(x_{k+1}-y_{k+1})\|^2 \\
    & +2 t_{k+1}\langle x_{k+1}-y_{k+1},t_{k+1}y_{k+1}-(t_{k+1}-1)x_k-x \rangle,
\end{align*}
where we used $t_k^2=t_{k+1}^2-t_{k+1}$ in view of $t_{k+1} := 0.5+\sqrt{0.25+t_k^2}$. This implies that 
\begin{align*}
     2L_\eta^{-1}\left(t_{k}^2v_k-t_{k+1}^2v_{k+1}\right) &\geq  \|t_{k+1}x_{k+1}-(t_{k+1}-1)x_k-x\|^2 \\
     & -\|t_{k+1}y_{k+1}-(t_{k+1}-1)x_k-x\|^2.
\end{align*}
Note that from Algorithm~\ref{alg:RAPM} we have $$t_{k+1}y_{k+1} =t_{k+1}x_k +(t_k-1)(x_k-x_{k-1}).$$ Invoking this and the definition of $u_k(x)$, from the preceding inequality, we obtain 
\begin{align*}
     2L_{\eta}^{-1}\left(t_{k}^2v_k-t_{k+1}^2v_{k+1}\right) &\geq  \|u_{k+1}(x)\|^2-\|u_{k}(x)\|^2.
\end{align*}
\end{proof}
Next, we provide a bound on the regularized function $F_\eta$. This result can be shown using Lemma~\ref{lem:RAPM3} following the steps taken in the proof of Theorem 4.4 in \cite{beck2009fast}. 
\begin{lemma}\label{lem:F_eta}
Consider Algorithm~\ref{alg:RAPM} and let Assumption \ref{assump:accel} hold. Let $x \in \mathbb{R}^n$ be an arbitrary vector. Then, for any $k\geq 1$,
\begin{align}\label{ineq:RAPM_main}
    F_\eta(x_k) -F_\eta(x) \leq \frac{2 (L_h+\eta L_f)\|x_0-x\|^2}{(k+1)^2}.
\end{align}
\end{lemma}

\section{Main Result} \label{section 3}
We are now ready to present the main convergence rate results for both the suboptimality and infeasibility metrics of Algorithm \ref{alg:1}.
\begin{theorem}\label{thm:RAPM}
    Consider Algorithm~\ref{alg:RAPM} and let Assumption~\ref{assump:accel} hold. Let $x^* \in \mathbb{R}^n$ be an optimal solution to problem~\eqref{prob:centr}. Suppose $\gamma  \leq \frac{1}{L_\eta}$, where $L_\eta:=L_h + \eta L_f$. Then, the following results hold for any $K\geq 1$. 

\noindent (i) [Suboptimality upper-bound] Let $X^*$ and $f^*$ be the optimal solution set and the optimal $f$ for~\eqref{prob:centr}, respectively. Then, 
\begin{align*}
    f(x_K) -f^* \leq \frac{2 (L_h+\eta L_f)\, \mbox{dist}^2(x_0,X^*)}{\eta (K+1)^2}.
\end{align*}

\noindent (ii) [Infeasibility bounds] Let $X^*_{\bar h} \triangleq \hbox{arg}\min_{x\in \mathbb{R}^n}\bar{h}(x)$, $\bar{h}^*\triangleq \inf_{x\in \mathbb{R}^n} \bar{h}(x)$. Suppose $X^*_{\bar h} $ is $\alpha$-weak sharp and $\eta \leq \frac{\alpha}{2\|\nabla f(x^*)\|}$. Then, we have
\begin{align*} 
    &0 \leq h(x_K) -h^*   \leq \frac{4 (L_h+\eta L_f)\mbox{dist}^2(x_0,X^*)}{(K+1)^2},
\end{align*}
which implies
\begin{align*}
      0 \leq   \mbox{dist}(x_K,X^*_{\bar{h}})   \leq  \frac{4 (L_h+\eta L_f)\, \mbox{dist}^2(x_0,X^*)}{\alpha \, (K+1)^2}.
\end{align*}

\noindent (iii) [Suboptimality lower-bound] Let the conditions in (ii) hold. Then, we have
\begin{align*}
    f(x_K) -f^* &\geq  - \|\nabla f(x^*)\| \left(\frac{4 (L_h+\eta L_f)\, \mbox{dist}^2(x_0,X^*)}{\alpha \, (K+1)^2}\right).
\end{align*}
\end{theorem}
\begin{proof}
\noindent (i) Consider the inequality \eqref{ineq:RAPM_main} for $K\geq 1$. Let us choose $x:=\Pi_{X^*}(x_0)$. Then, noting that $\Pi_{X^*}(x_0) \in X^*_{\bar{h}}$, we have $h(x_K) - h(x)\geq 0$. This, together with \eqref{ineq:RAPM_main}, imply that 
\begin{align*} 
    \eta(f(x_K) -f(x)) \leq \frac{2 (L_h+\eta L_f)\|x_0-x\|^2}{(K+1)^2}.
\end{align*}
This implies the suboptimality upper-bound in (i).

\noindent (ii) The lower bound holds by definition of $\bar{h}^*$. To show the upper bound,  let us define $\hat{x}_K \triangleq \Pi_{X^*_{\bar{h}}}(x_K)$. Note that $\hat{x}_K  \in X^*_{\bar{h}} $ by definition. In view of the optimality condition for the upper-level problem, we have $\langle\nabla f(x^*),\hat{x}_K-x^*\rangle \geq 0$. From the convexity of $f$ and the Cauchy-Schwarz inequality, we have
\begin{align}\label{eqn:lower_bound_dist}
f(x_K) - f^* & \geq \langle\nabla f(x^*),x_K-x^*\rangle \notag
\\
&=\langle\nabla f(x^*),x_K-\hat{x}_K\rangle+\langle\nabla f(x^*),\hat{x}-x^*\rangle \notag\\ 
&\geq \langle\nabla f(x^*),x_K-\hat{x}_K\rangle \notag\\
&\geq -\|\nabla f(x^*)\|  \mbox{dist}(x_K,X^*_{\bar{h}}).
\end{align}
Consider \eqref{ineq:RAPM_main} for $k:=K$ and $x:=\Pi_{X^*}(x_0)$. Substituting the preceding bound in \eqref{ineq:RAPM_main}, we obtain 
\begin{align*} 
    &h(x_K) -h^* -\eta\|\nabla f(x^*)\|  \mbox{dist}(x_K,X^*_{\bar{h}}) \leq \\
    &\frac{2 (L_h+\eta L_f)\mbox{dist}^2(x_0,X^*)}{(K+1)^2}.
\end{align*}
From the $\alpha$-weak sharp property of $X^*_{\bar{h}}$, we obtain

\begin{align*} 
    &(h(x_K) -h^*) -\eta\|\nabla f(x^*)\|  \frac{ (h(x_K) -h^*)}{\alpha}   \leq \\
&\frac{2 (L_h+\eta L_f)\mbox{dist}^2(x_0,X^*)}{(K+1)^2}
.
\end{align*}

In view of $\eta \leq \frac{\alpha}{2\|\nabla f(x^*)\|}$, we obtain 

\begin{align*} 
    &h(x_K) -h^*   \leq \frac{4 (L_h+\eta L_f)\mbox{dist}^2(x_0,X^*)}{(K+1)^2}.
\end{align*}
Also, regarding Definition~\ref{def:weaksharp} we have
\begin{align*} 
    &0 \leq   \alpha\mbox{dist}(x_K,X^*_{\bar{h}})  \leq h(x_K) -h^*.
\end{align*}
Therefore, we obtain the desired results.

%From the preceding relation, the weak-sharpness property of $X^*_{\bar{h}}$, and the bound in (ii), we have obtain the bound in (iii).  
\noindent (iii) This follows from combining \eqref{eqn:lower_bound_dist} and the inequality in part (ii).
\end{proof}
\begin{remark}[Complexity guarantees] \label{Remark:1} Note that Theorem~\ref{thm:RAPM} parts (i) and (ii) provide the compactly presented rate result of suboptimality and infeasibility, as follows.
$$|f(x_K)-f^*| \leq \frac{ 2 (L_h \eta^{-1}+ L_f)\, \mbox{dist}^2(x_0,X^*) }{(K+1)^2},$$
and
\begin{align*} 
    &h(x_K) -h^*   \leq \frac{4 (L_h+\eta L_f)\mbox{dist}^2(x_0,X^*)}{(K+1)^2}.
\end{align*}
Let $\epsilon_f>0$ be an arbitrary scalar such that $|f(x_K)-f^*| \leq \epsilon_f$. Hence, the number of iterations to guarantee this accuracy regarding suboptimality is as follows.
\begin{align*}
K & \geq \frac{\sqrt{2(L_h \eta^{-1}+ L_f)}\, \text{dist}(x_0,X^*)}{\sqrt{\epsilon_f}} - 1.
\end{align*}
Now, let $\epsilon_h>0$ be an arbitrary scalar such that $h(x_K) -h^*   \leq \epsilon_h$. Hence, the number of iterations to guarantee this accuracy regarding infeasibility is as follows.
\begin{align*}
K & \geq \frac{2  \sqrt{(L_h+\eta L_f)}\, \text{dist}(x_0,X^*)}{\sqrt{\epsilon_h}} - 1.
\end{align*}

Let $\epsilon=\epsilon_f=\epsilon_h >0$ be an arbitrary scalar such that $\mbox{max}\left(|f(x_K)-f^*|, h(x_K) -h^*  \right) \leq \epsilon$. 
Therefore, the number of iterations of R-APM to guarantee that both suboptimality and infeasibility to have this accuracy is
\begin{align*}
K & \geq \max\left(\frac{\sqrt{2(L_h \eta^{-1}+ L_f)}\, \text{dist}(x_0,X^*)}{\sqrt{\epsilon}} - 1, \right.\\
& \left. \frac{2  \sqrt{(L_h+\eta L_f)}\, \text{dist}(x_0,X^*)}{\sqrt{\epsilon}} - 1\right).
\end{align*}
Therefore, the iteration complexity  for R-APM for addressing problem \ref{prob:centr} is $\mathcal{O}(\epsilon^{-0.5})$.
\end{remark}
A key assumption leveraged in the analysis of the result in Theorem~\ref{thm:RAPM} is the weak-sharpness of the lower-level optimal solution set. Also, the threshold for $\eta$ depends on the weak-sharp parameter. When either the weak-sharp property is not met, or the parameter $\alpha$ is unavailable, a natural question is how to implement the proposed method. This is precisely addressed in the following. 
\begin{proposition}\label{propositon:RAPM}
    Consider Algorithm~\ref{alg:RAPM} and let Assumption~\ref{assump:accel} hold. Let $x^* \in \mathbb{R}^n$ be an optimal solution to problem~\eqref{prob:centr}. Suppose $\gamma  \leq \frac{1}{L_\eta}$, $\eta: = \frac{1}{K+1}$, where $L_\eta:=L_h + \eta L_f$. Then, the following results hold for any $K\geq 1$. 

\noindent (i) [Suboptimality upper-bound] Let $X^*$ and $f^*$ be the optimal solution set and the optimal $f$ for problem~\eqref{prob:centr}, respectively. Then, 
\begin{align*}
    f(x_K) -f^* \leq \frac{2 L_f\ \mbox{dist}^2(x_0,X^*)}{(K+1)^2}+ \frac{2 L_h\ \mbox{dist}^2(x_0,X^*)}{K+1}.
\end{align*}

\noindent (ii) [Infeasibility bounds] Let $X^*_{\bar h} \triangleq \hbox{arg}\min_{x\in \mathbb{R}^n}\bar{h}(x)$, $\bar{h}^*\triangleq \inf_{x\in \mathbb{R}^n} \bar{h}(x)$, and ${D}_f \triangleq \sup_{k=1,\ldots, K} (f(\Pi_{X^*_{\bar{h}}}(x_0)) - f(x_k))$. Then, we have
\begin{align*}
    0 \leq \bar{h}(x_K) -\bar{h}^* &\leq \frac{2 L_f\ \mbox{dist}^2(x_0,X^*_{\bar{h}})}{(K+1)^3}+ \frac{2 L_h\ \mbox{dist}^2(x_0,X^*_{\bar{h}})}{(K+1)^2}\\
    &+ \frac{{D}_f}{K+1}.
\end{align*}

\end{proposition}
\begin{proof}
\noindent (i) %Consider \eqref{ineq:RAPM_main} for $K$. Let us choose $x:=\Pi_{X^*}(x_0)$. Then, noting that $\Pi_{X^*}(x_0) \in X^*_{\bar{h}}$, we have $h(x_K) - h(x)\geq 0$. This together with \eqref{ineq:RAPM_main} imply that 
From Theorem~\ref{thm:RAPM} (i), we have
\begin{align*} 
    \eta(f(x_K) -f(x)) \leq \frac{2 (L_h+\eta L_f)\|x_0-x\|^2}{(K+1)^2}.
\end{align*}
From $\eta:=\frac{1}{K+1}$, $L_\eta:=L_h+\eta L_f$, and the definition of the distance function, we obtain the bound in (i).

\noindent (ii) The lower bound holds by definition of $\bar{h}^*$. To show the upper bound, consider \eqref{ineq:RAPM_main} for $K$ once again, but this time substitute $x:=\Pi_{X^*_{\bar{h}}}(x_0)$. 
\end{proof}
\begin{remark}[Constant ${D}_f$] Note that ${D}_f$ in Proposition~\ref{propositon:RAPM} does not necessarily grow with $K$. This is because the boundedness of the term $\limsup_{ K\to \infty} (f(\Pi_{X^*_{\bar{h}}}(x_0)) - f(x_k))$ can be guaranteed under several settings. For example, when $f(x) >- \infty$ for all $x$, this is the case. Another instance is when $\omega(x)$ is the indicator function of a compact set $X$, implying that, $x_k$ remains in the set. This is because the proximal map of $\gamma\omega$ in step \# 3 of Algorithm~\ref{alg:RAPM} will coincide with the Euclidean projection map onto the set $X$. 
\end{remark}
\begin{remark}[Complexity guarantees] Notably, Proposition~\ref{propositon:RAPM} provides guarantees for the upper-level metric (only an upper-bound) and the lower-level (both lower and upper bounds) in the absence of the weak-sharpness property.  Let $\epsilon=\epsilon_f=\epsilon_h >0$ be an arbitrary scalar such that $\mbox{max}\left(f(x_K)-f^*, h(x_K) -h^*  \right) \leq \epsilon$. Then, same as the discussion in Remark \ref{Remark:1} we reach to the iteration complexity for R-APM for addressing problem \eqref{prob:centr} that is $\mathcal{O}(\epsilon^{-1})$. Notably, this is still the best known complexity in the literature in the absence of the weak-sharpness property. 
\end{remark}
%
%================================================
%======= Numerical examples
%================================================
\section{Numerical Experiments }\label{sec: 4}
We consider the problem of solving a sparse linear regression problem of the form~\cite{jiang2022conditional}
 \begin{equation} \label{num:prob}
%\tag{$1$}  
\begin{split}
& \min \quad f(x) \triangleq \frac{1}{2}\|{\bf{A}_{\text{val}}}(x) - b_{\text{val}}\|^2 \\
&\textrm {s.t.} \quad  x \in \argmin_{x \in \mathbb{R}^n} \{\bar{h}(x)\triangleq  \ell_{\text{tr}}(x) + \delta_X(x)  \},
\end{split}
\end{equation}
where $\ell_{\text{tr}} =\frac{1}{2}\|{\bf{A}_{\text{tr}}}(x) - b_{\text{tr}}\|^2 $ is a loss function and $\delta_X$ is an indicator function over set $X = \{x \,|\, \|x\|_1 \leq \lambda\}$, which is the feasible set for the lower-level optimization problem given by
 \begin{equation} \label{num:prob-lower}
\begin{split}
& \min \quad h(x) \triangleq \ell_{\text{tr}}(x),\quad\textrm {s.t.} \quad  x \in X.
\end{split}
\end{equation}
Here, the matrices $\bf{A}_{\text{tr}}$ and $\bf{A}_{\text{val}}$, and vectors ${b}_{\text{tr}}$ and ${b}_{\text{val}}$ are of appropriate dimensions. Note that the indicator function $\delta_X(x)$ represents the $\omega$ function in \eqref{prob:centr}.
The goal of this sparse regression problem is to find the parameter vector $x$ minimizing the loss function $f$ with respect to a training dataset (${\bf{A}_{\text{tr}}},b_{\text{tr}}$). It is important to note that this class of regression problems has multiple global minima. Therefore, the second objective function $h$, which is minimized over a validation dataset (${\bf{A}_{\text{val}}},b_{\text{val}}$), is utilized to select one of the minimizers of the training set. 

Following~\cite{jiang2022conditional}, we utilize the Wikipedia Math Essential dataset, which is divided into a 60\% validation subset 20\% training subset, and the remaining data is used for testing.
We consider two feasible sets for the lower-level optimization problem \eqref{num:prob-lower} given by $X_1 = \{x \,|\, \|x\|_1 \leq 1\}$ and $X_2 = \{x \,|\, \|x\|_1 \leq 2\}$.

We implement Algorithm~\ref{alg:1} by setting $\eta = \frac{1}{K+1}$ and  $\gamma \leq 1/ \left(L_h + \eta L_f\right) $, where $L_h=\lambda_{\text{max}}({ {\bf{A}}_{\text{tr}}}^T{\bf{A}}_{\text{tr}}) $ and $L_f=\lambda_{\text{max}}({{ \bf{A}}_{\text{val}}}^T{\bf{A}}_{\text{val}}) $, are Lipschitz parameters of functions $h$ and $f$, respectively. Note that the choice $\eta = \frac{1}{K+1}$ is done in view of unavailability of the weak-sharpness parameter (see Prop.~\ref{propositon:RAPM}).  
We compare the performance of our method with existing ones comprising MNG \cite{beck2014first}, BiG-SAM \cite{sabach2017first},  a-IRG \cite{kaushik2021method}, and CG-BiO \cite{jiang2022conditional}. For the numerical implementation of these methods, we use the Matlab codes from~\cite{jiang2022conditional} available in \cite{cgbio}.
We quantify the infeasibility and suboptimality through the performance metrics $|{\bar{h}}(x_k)-{\bar{h}}^*|$ and $|f(x_k)-f^*|$, respectively. Here, $x_k$ denotes the solution iterates of the algorithm, while ${\bar{h}}^*$ and ${f}^*$ represent the baseline optimal values found using the CVX method \cite{cvx}.
The results are shown in Figure \ref{results} for both feasible sets $X_1$ and $X_2$. Note that our method outperforms existing ones not only in terms of convergence rate but also in accuracy.

\begin{figure*}\label{1}
\centering
\begin{table}[H]
\centering
\setlength{\tabcolsep}{0.000001pt} % Adjust this value to reduce or increase the space between columns
\begin{tabular}{c|cc}
{ {$X$ in problem \eqref{num:prob}}\ \ }& { Suboptimality} & { Infeasibility}  \\
\hline\\
\rotatebox[origin=c]{90}{{ {$X_1= \{x \,|\, \|\mathbf{x}\|_1 \leq 1\}$}}}
&
\begin{minipage}{0.4\textwidth} % Adjust the width as needed
\includegraphics[width=\textwidth, angle=0]{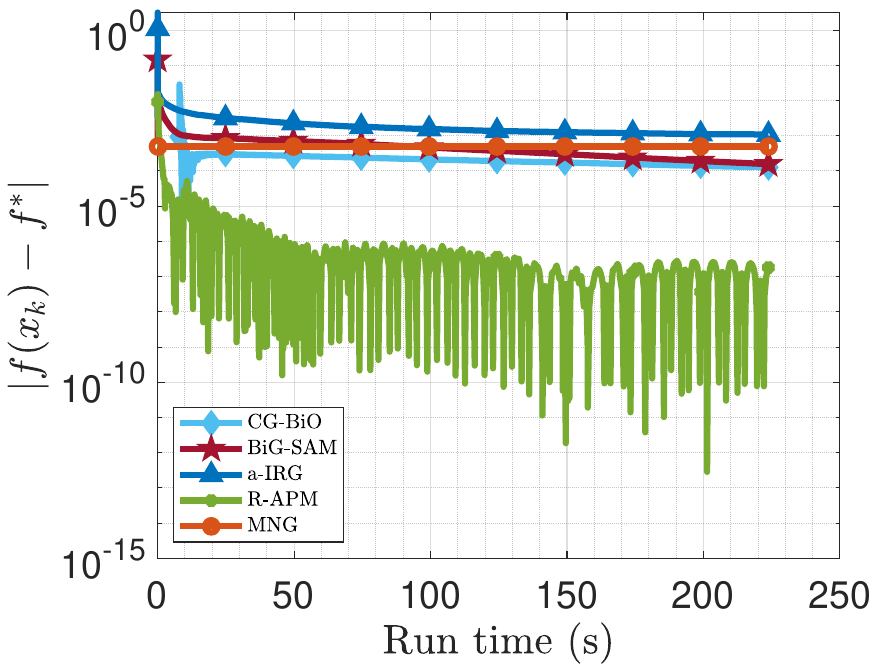}
\end{minipage}
&
\begin{minipage}{0.4\textwidth} % Adjust the width as needed
\includegraphics[width=\textwidth, angle=0]{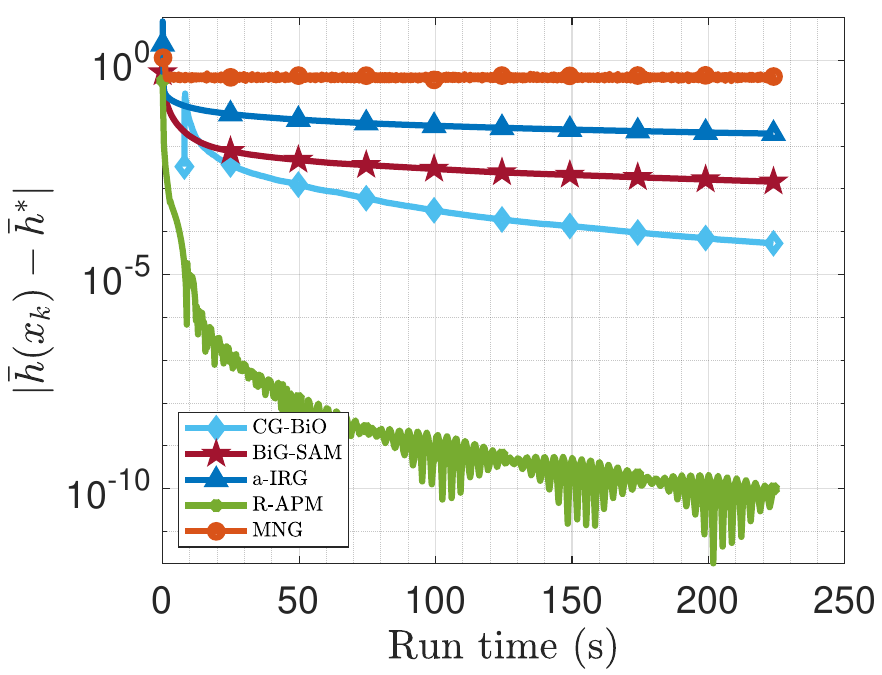}
\end{minipage}
\\
\hline\\
\rotatebox[origin=c]{90}{{\footnotesize {$X_2 = \{x \,|\, \|\mathbf{x}\|_1 \leq 2\}$}}}
&
\begin{minipage}{0.4\textwidth} % Adjust the width as needed
\includegraphics[width=\textwidth, angle=0]{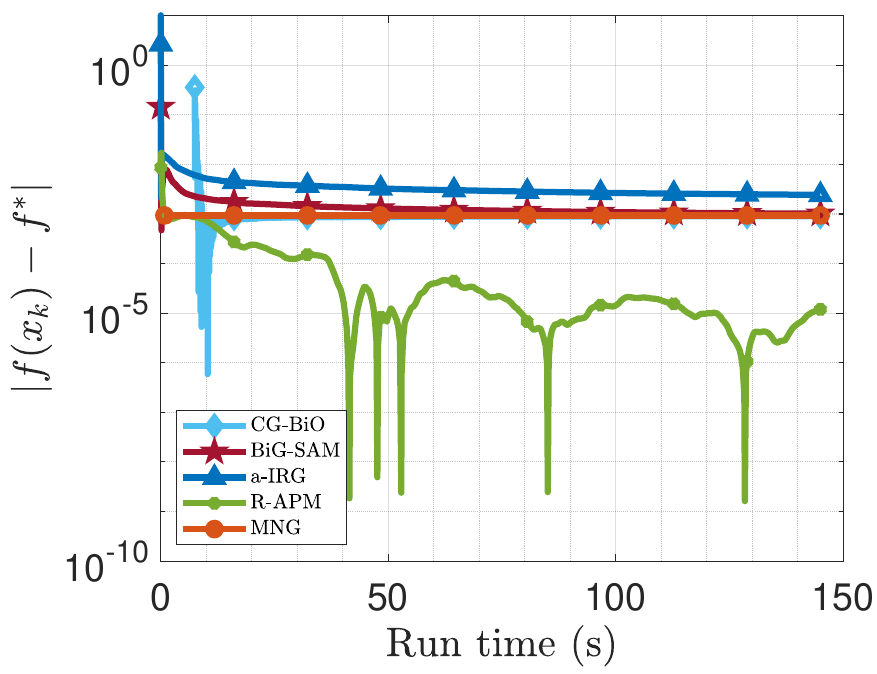}
\end{minipage}
&
\begin{minipage}{0.4\textwidth} % Adjust the width as needed
\includegraphics[width=\textwidth, angle=0]{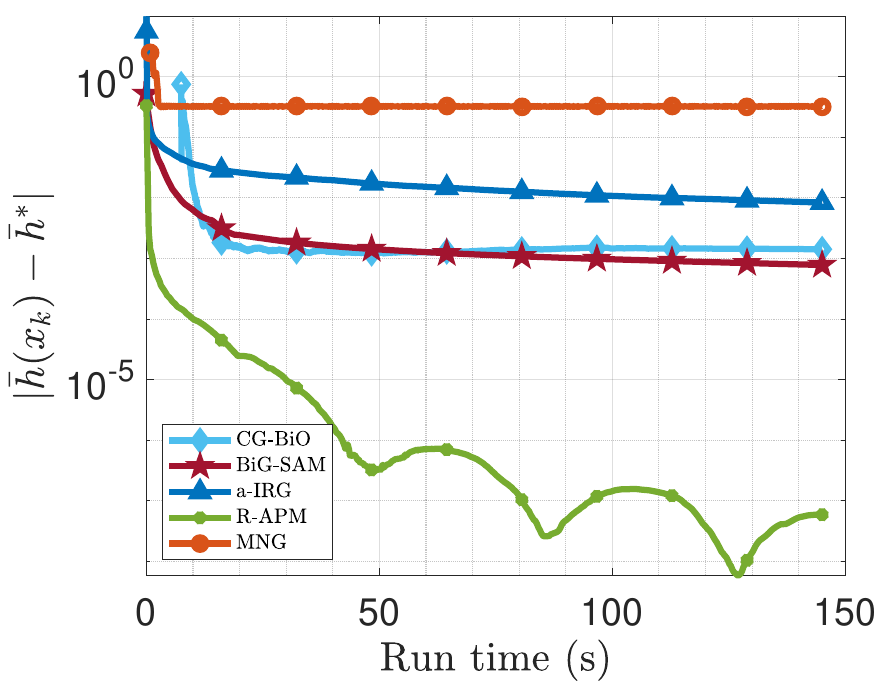}
\end{minipage}
\end{tabular}
\end{table}

\caption{Comparison of methods including R-APM, CG-BiO, BiG-SAM, a-IRG, and MNG  in terms of suboptimality and infeasibility in solving problem \eqref{num:prob} for two different feasible sets.}
\label{results}
\end{figure*}

\section{Conclusions and Future work} \label{section 5}
In this work, we proposed a novel accelerated algorithm to solve simple bilevel composite optimization problems with convex objective functions in both upper and lower-levels. The proposed algorithm leverages the structure of the FISTA method by incorporating a regularized gradient mapping that depends on the weighted information of the gradient of the
lower-level and upper-level objective functions. Remarkably, our analysis demonstrates that the proposed method achieves a convergence rate of $\mathcal{O}(\epsilon^{-0.5})$, thereby matching the fastest-known convergence rate for single-level optimization problems. This results in the fastest known algorithm among existing methods to solve simple bilevel problems. 
The effectiveness of our approach is illustrated through the practical application example of solving sparse linear regression problems. When compared to existing methods, our results indicate that our proposed algorithm not only converges faster but also exhibits a significantly higher accuracy. Future work aims to extend our results to the case where the upper and lower-level functions are non-convex and where multiple agents cooperate to solve a bilevel optimization problem.

\bibliographystyle{IEEEtran}
\bibliography{references}

\begin{thebibliography}{10}
\providecommand{\url}[1]{#1}
\csname url@samestyle\endcsname
\providecommand{\newblock}{\relax}
\providecommand{\bibinfo}[2]{#2}
\providecommand{\BIBentrySTDinterwordspacing}{\spaceskip=0pt\relax}
\providecommand{\BIBentryALTinterwordstretchfactor}{4}
\providecommand{\BIBentryALTinterwordspacing}{\spaceskip=\fontdimen2\font plus
\BIBentryALTinterwordstretchfactor\fontdimen3\font minus
  \fontdimen4\font\relax}
\providecommand{\BIBforeignlanguage}[2]{{%
\expandafter\ifx\csname l@#1\endcsname\relax
\typeout{** WARNING: IEEEtran.bst: No hyphenation pattern has been}%
\typeout{** loaded for the language `#1'. Using the pattern for}%
\typeout{** the default language instead.}%
\else
\language=\csname l@#1\endcsname
\fi
#2}}
\providecommand{\BIBdecl}{\relax}
\BIBdecl

\bibitem{friedlander2008exact}
M.~P. Friedlander and P.~Tseng, ``Exact regularization of convex programs,''
  \emph{SIAM Journal on Optimization}, vol.~18, no.~4, pp. 1326--1350, 2008.

\bibitem{beck2014first}
A.~Beck and S.~Sabach, ``A first order method for finding minimal norm-like
  solutions of convex optimization problems,'' \emph{Mathematical Programming},
  vol. 147, no. 1-2, pp. 25--46, 2014.

\bibitem{amini2019iterative}
M.~Amini and F.~Yousefian, ``An iterative regularized incremental projected
  subgradient method for a class of bilevel optimization problems,'' in
  \emph{2019 American Control Conference (ACC)}.\hskip 1em plus 0.5em minus
  0.4em\relax IEEE, 2019, pp. 4069--4074.

\bibitem{merchav2023convex}
R.~Merchav and S.~Sabach, ``Convex bi-level optimization problems with
  non-smooth outer objective function,'' \emph{arXiv preprint
  arXiv:2307.08245}, 2023.

\bibitem{beck2009fast}
A.~Beck and M.~Teboulle, ``A fast iterative shrinkage-thresholding algorithm
  for linear inverse problems,'' \emph{SIAM journal on imaging sciences},
  vol.~2, no.~1, pp. 183--202, 2009.

\bibitem{solodov2007explicit}
M.~Solodov, ``An explicit descent method for bilevel convex optimization,''
  \emph{Journal of Convex Analysis}, vol.~14, no.~2, p. 227, 2007.

\bibitem{solodov2007bundle}
\BIBentryALTinterwordspacing
M.~V. Solodov, ``A bundle method for a class of bilevel nonsmooth convex
  minimization problems,'' \emph{SIAM Journal on Optimization}, vol.~18, no.~1,
  pp. 242--259, 2007. [Online]. Available:
  \url{https://doi.org/10.1137/050647566}
\BIBentrySTDinterwordspacing

\bibitem{helou2017}
E.~S. Helou and L.~E. Sim{\~o}es, ``$\epsilon$-subgradient algorithms for
  bilevel convex optimization,'' \emph{Inverse Problems}, vol.~33, no.~5, p.
  055020, 2017.

\bibitem{malitsky2017chambolle}
Y.~Malitsky, ``Chambolle-pock and tseng’s methods: relationship and extension
  to the bilevel optimization,'' \emph{arXiv preprint arXiv:1706.02602}, p.~3,
  2017.

\bibitem{sabach2017first}
S.~Sabach and S.~Shtern, ``A first order method for solving convex bilevel
  optimization problems,'' \emph{SIAM Journal on Optimization}, vol.~27, no.~2,
  pp. 640--660, 2017.

\bibitem{kaushik2021method}
H.~D. Kaushik and F.~Yousefian, ``A method with convergence rates for
  optimization problems with variational inequality constraints,'' \emph{SIAM
  Journal on Optimization}, vol.~31, no.~3, pp. 2171--2198, 2021.

\bibitem{jiang2022conditional}
R.~Jiang, N.~Abolfazli, A.~Mokhtari, and E.~Y. Hamedani, ``Conditional
  gradient-based method for bilevel optimization with convex lower-level
  problem,'' in \emph{OPT 2022: Optimization for Machine Learning (NeurIPS 2022
  Workshop)}, 2022.

\bibitem{shen2023online}
L.~Shen, N.~Ho-Nguyen, and F.~K{\i}l{\i}n{\c{c}}-Karzan, ``An online convex
  optimization-based framework for convex bilevel optimization,''
  \emph{Mathematical Programming}, vol. 198, no.~2, pp. 1519--1582, 2023.

\bibitem{tikhonov1963solution}
A.~N. Tikhonov, ``On the solution of ill-posed problems and the method of
  regularization,'' in \emph{Doklady akademii nauk}, vol. 151, no.~3.\hskip 1em
  plus 0.5em minus 0.4em\relax Russian Academy of Sciences, 1963, pp. 501--504.

\bibitem{yousefian2021bilevel}
F.~Yousefian, ``Bilevel distributed optimization in directed networks,'' in
  \emph{2021 American Control Conference (ACC)}.\hskip 1em plus 0.5em minus
  0.4em\relax IEEE, 2021, pp. 2230--2235.

\bibitem{gong2021bi}
C.~Gong and X.~Liu, ``Bi-objective trade-off with dynamic barrier gradient
  descent,'' \emph{NeurIPS 2021}, 2021.

\bibitem{Beck2017}
A.~Beck, \emph{First-Order Methods in Optimization}.\hskip 1em plus 0.5em minus
  0.4em\relax Philadelphia, PA: MOS-SIAM Series on Optimization, Society for
  Industrial and Applied Mathematics (SIAM), 2017.

\bibitem{studniarski1999weak}
M.~Studniarski and D.~E. Ward, ``Weak sharp minima: characterizations and
  sufficient conditions,'' \emph{SIAM Journal on Control and Optimization},
  vol.~38, no.~1, pp. 219--236, 1999.

\bibitem{cgbio}
Raymond30, ``{CG-BiO},'' \url{https://github.com/Raymond30/CG-BiO}.

\bibitem{cvx}
M.~Grant and S.~Boyd, ``{CVX: Matlab software for disciplined convex
  programming, version 2.1},'' \url{http://cvxr.com/cvx}, Mar. 2014.

\end{thebibliography}

\end{document}